\documentclass[12pt]{article} 

\usepackage{amsmath, amsthm, amssymb, mathrsfs, fullpage} 
\usepackage[all]{xy} 
\input xy 

\xyoption{all}
\xyoption{2cell} 
\xyoption{v2}

\usepackage{hyperref}

\renewcommand{\a}{\alpha}
\renewcommand{\b}{\beta}
\renewcommand{\phi}{\varphi}

\newcommand{\cC}{\mathscr{C}}

\newcommand{\Dec}{\mathrm{Dec}}
\newcommand{\sk}{\mathrm{sk}}
\newcommand{\cosk}{\mathrm{cosk}} 
 
\newcommand{\Wbar}{\overline{W}} 
\newcommand{\disc}{\mathrm{disc}} 

\newcommand{\coeq}{\mathrm{coeq}} 
\newcommand{\op}{\mathrm{op}} 
\newcommand{\id}{\mathrm{id}} 
\newcommand{\Set}{\mathbf{Set}} 
\newcommand{\Gp}{\mathbf{Gp}} 
\newcommand{\Ab}{\mathbf{Ab}}

\newcommand{\Ch}{\mathbf{Ch}} 
\renewcommand{\S}{\mathbf{S}}
\renewcommand{\SS}{\mathbf{SS}} 
\newcommand{\SGpd}{\mathbf{SGpd}} 

\newcommand{\Hom}{\mathrm{Hom}} 
\newcommand{\last}{\mathrm{last}} 
\newcommand{\first}{\mathrm{first}} 
\newcommand{\Tot}{\mathrm{Tot}}

\theoremstyle{plain}
\newtheorem{theorem}{Theorem}
\newtheorem{lemma}[theorem]{Lemma}
\newtheorem{proposition}[theorem]{Proposition}
\newtheorem{corollary}[theorem]{Corollary}

\theoremstyle{definition}
\newtheorem{definition}[theorem]{Definition}

\theoremstyle{definition}

\theoremstyle{remark}

\title{D\'{e}calage and Kan's simplicial loop group functor} 
\author{Danny Stevenson\thanks{The author was supported by the 
 Engineering and Physical Sciences Research Council [grant number EP/I010610/1]}\\
School of Mathematics and Statistics\\
University of Glasgow  \\
15 University Gardens\\
Glasgow G12 8QW\\
United Kingdom\\
email: {\tt Danny.Stevenson@glasgow.ac.uk} }

\begin{document} 
\maketitle
\begin{abstract}
Given a bisimplicial set, there are two ways to extract 
from it a simplicial set: the diagonal simplicial set and the less 
well known total simplicial set of Artin and Mazur.  
There is a natural comparison map between these 
simplicial sets, and it is a theorem due to Cegarra and 
Remedios and independently Joyal and Tierney, that this comparison map is a weak 
equivalence for any bisimplicial set.  In this paper we will give a new, elementary 
proof of this result.  As an application, we will revisit 
Kan's simplicial loop group functor $G$.  We will 
give a simple formula for this functor, which is based on a 
factorization, due to Duskin, of Eilenberg and Mac Lane's 
classifying complex functor $\overline{W}$. 
We will give a new, short, proof of Kan's result that the unit 
 map for the adjunction $G\dashv \overline{W}$ is a weak equivalence 
for reduced simplicial sets.       
\newline 

\end{abstract}

\tableofcontents

\section{Introduction} 

The aim of this paper is to give new and hopefully simpler proofs of two theorems in 
the theory of simplicial sets, the first being a generalization to simplicial sets of 
Dold-Puppe's version \cite{DP} of the Eilenberg-Zilber theorem from homological algebra, 
the second being an old result of Kan's on simplicial loop groups.  This 
first result is due to Cegarra and Remedios and independently to 
Joyal and Tierney.  

Recall that if $C$ is a double complex of abelian groups concentrated in the first 
quadrant then there are two ways in which 
one can associate to it an ordinary complex.  One can form the total complex 
$\Tot\, C$ which in degree $n$ is equal to 
\[
(\Tot\, C)_n = \bigoplus_{p+q=n} C_{p,q},
\]
or one can form the diagonal complex $dC$ which in degree $n$ is equal to 
\[
(dC)_n = C_{n,n}.  
\]
There is a natural comparison map $dC\to \Tot\, C$ and the generalized Eilenberg-Zilber theorem \cite{DP, GJ} 
says that this comparison map is a chain homotopy equivalence.  

There is a generalization of this comparison with 
chain complexes, or equivalently simplicial abelian groups, 
replaced by simplicial sets.  Just as we can 
form the diagonal of a double complex 
we can also form the diagonal $dX$ of a {\em bisimplicial set} $X$.    
This is the simplicial 
set obtained by precomposing the functor 
$X\colon \Delta^\op\times \Delta^\op\to \Set$ 
with the opposite of the functor $\delta\colon 
\Delta\to \Delta\times \Delta$ given by $\delta([n]) =  ([n],[n])$, i.e.\ 
$dX = X\delta$, so that the set of $n$-simplices of $dX$ is 
$(dX)_n = X_{n,n}$.  
There is another, less well known, way to 
form a simplicial set from $X$.  Namely one can form 
what is variously known as the {\em total 
simplicial set} or {\em Artin-Mazur codiagonal} 
$T X$ of $X$ (see \cite{AM}).  This construction 
extends to define a functor $T\colon \SS\to \S$ 
from the category $\SS$ of bisimplicial sets to the category $\S$ of 
simplicial sets.  

The construction $T X$ is the analog for simplicial sets of the process of forming the total 
complex $\Tot\,  C$ of a double complex $C$.  In fact, if 
\[
N\colon s\Ab\leftrightarrows \Ch_{\geq 0}\colon \Gamma 
\]
denotes the Dold-Kan correspondence, then $NT A$ is isomorphic to 
$\Tot\, NA$ (see \cite{CR}).  

As mentioned above, the Eilenberg-Zilber theorem in homological algebra has a 
generalization for simplicial sets.  For any bisimplicial set
$X$, there is a natural comparison map 
\[
dX\to T X
\]
between the diagonal simplicial set of $X$ and the total simplicial set of $X$.  We have the 
following result.  
\begin{theorem}[\cite{CR}, \cite{JT2}]
\label{CR thm}
Let $X$ be a bisimplicial set.  Then the comparison map $dX\to T X$ is a weak equivalence.  
\end{theorem}

The first published proof of this result was given in \cite{CR}, with the authors 
noting that this fact is stated without proof in \cite{Cordier} where it is attributed to Zisman (unpublished).  
When $X$ is a bisimplicial group, a closely related result was proven by 
Quillen \cite{Quillen}.   
The proof of Theorem~\ref{CR thm} given in \cite{CR} is unfortunately somewhat complicated, and so it is of interest to have a simpler proof.  
One such proof, incorporating some ideas of Cisinski, is given by Joyal and Tierney in their forthcoming book \cite{JT2}.  
We shall give here a new proof, which we think is fairly elementary --- in particular it uses 
nothing more than the fact that the diagonal functor $d$ sends level-wise weak equivalences to 
weak equivalences.   

In the second part of the paper we present a simple construction 
of Kan's simplicial loop group functor as an application of Theorem~\ref{CR thm}.     
Recall that in \cite{Kan1}, Kan defined a functor $G\colon \S\to s\Gp$ 
which is left adjoint to the classifying complex functor 
$\overline{W}\colon s\Gp\to \S$ of Eilenberg and Mac Lane 
\cite{EM}.   He was able to show that, when $X$ is reduced (i.e.\  when
$X$ is a simplicial set with only one vertex), the principal $GX$ bundle 
$X_\eta$ on $X$ induced by the unit map $\eta\colon X\to \overline{W}GX$ has weakly 
contractible total space.  Kan's proof of this last fact involves 
showing firstly that $X_\eta$ is simply connected, and 
secondly that $X_\eta$ is acyclic in the sense that 
it has vanishing reduced homology in all degrees.  

We will show that both the construction of the functor $G$ 
and the proof that the unit map is a weak equivalence can be greatly illuminated 
and simplified by considering a factorization (first noticed by Duskin) of $\overline{W}$ involving the  
functor $T$.  In fact we hasten to point out that this last section of the 
paper makes no great claim to originality, we find it hard to believe 
that some of the results of this section were not known to Duskin, although we 
cannot find any evidence for this in his published papers.  We also point out that in 
their forthcoming book \cite{JT2} and their paper \cite{JT3} Joyal and Tierney 
prove more general statements in the context of simplicially enriched groupoids.  
Using Duskin's factorization we will give a simple formula for the left adjoint 
to $\overline{W}$ (see Proposition~\ref{left adjoint for Wbar}).  
In Theorem~\ref{Kan's thm} we will apply this formula to give a simple and 
direct proof of Kan's theorem that 
the unit map of the adjunction $G\dashv \overline{W}$ is a 
weak equivalence whenever $X$ is reduced.  To the best of our knowledge 
this proof is new (we note that an essential ingredient for the proof 
is Theorem~\ref{CR thm}).  We point out          
that in \cite{Wa} Waldhausen 
described another approach to the construction of $G$, 
nevertheless we feel our approach (which proceeds along 
different lines) is still of some interest.

\section{The d\'{e}calage comonad}
\label{app:one}

We begin by recalling the definition and main properties of 
the d\'{e}calage and total d\'{e}calage functors of Illusie \cite{Illusie}.

\subsection{The d\'{e}calage or shift functor} 

Let $\Delta_a $ denote the augmented simplex category, 
in other words the simplex category $\Delta$ together with 
the additional object $[-1]$, the empty set (the initial 
object of $\Delta_a $).  We will write 
$as\cC$ for the category $[\Delta_a ^\op,\cC]$ of augmented 
simplicial objects in a category $\cC$, which we will assume to be complete and cocomplete.     
Recall (see for example VII.5 of \cite{CWM}) that $\Delta_a $ is a monoidal 
category with unit $[-1]$ under the operation of ordinal 
sum, which operation we will denote by $\sigma$ (following Joyal and 
Tierney).  If $[m]$, $[n]\in \Delta_a$ then $\sigma([m],[n]) = [m+n+1]$, and the operation    
$\sigma$ gives rise to a bifunctor $\sigma \colon \Delta_a\times \Delta_a\to \Delta_a$ 
which sends a morphism 
\[
(\a,\b)\colon ([m],[n])\to ([m'],[n'])
\]
in $\Delta_a\times \Delta_a$ 
to the morphism $\sigma(\a, \b)\colon [m+n+1]\to [m'+n'+1]$ 
in $\Delta_a$ defined by 
\[
\sigma(\a, \b)(i) = \begin{cases} 
\a(i) & \text{if}\ 0\leq i\leq m \\ 
\b(i - m - 1) + m' + 1 & \text{if}\ m+1\leq i\leq m+n+1. 
\end{cases} 
\]
$(\Delta_a,\sigma)$ is not a symmetric monoidal category --- 
while $\sigma([m], [n]) = \sigma([n], [m])$, it need not be the case 
that $\sigma(\a, \b) = \sigma(\b, \a)$.  
The monoidal structure on $\Delta_a $ allows us to 
define a functor $\sigma(-, [0])\colon \Delta_a \to \Delta$ 
which sends $[n]\in \Delta_a $ to $\sigma([n], [0]) = [n+1]$ in 
$\Delta$.  
We have the following definition which we believe is originally due 
to Illusie.  
\begin{definition}[\cite{Illusie}] 
Define $\Dec_0\colon s\cC\to as\cC$ to be the functor 
given by restriction along $\sigma(-, [0])\colon \Delta_a\to \Delta$, so that 
if $X$ is a simplicial object in $\cC$ then $\Dec_0X$ is the augmented simplicial object given by
\[ 
\Dec_0X([n]) = X([n+1]),  
\] 
whose face maps $d_i\colon (\Dec_0 X)_n\to 
(\Dec_0 X)_{n-1}$ are given by $d_i\colon X_{n+1}\to X_n$ for 
$i=0,1,\ldots, n$, and whose degeneracy maps 
$s_i\colon (\Dec_0 X)_n\to (\Dec_0 X)_{n+1}$ are given by 
$s_i\colon X_{n+1}\to X_{n+2}$ for $i=0,1,\ldots, n$.  
The augmentation $(\Dec_0 X)_0\to X_0$ is given 
by $d_0\colon X_1\to X_0$.  

$\Dec_0X$ is obtained from $X$ by forgetting the top face and degeneracy map at each 
level and re-indexing by shifting degrees up by one.  
Thus the augmented simplicial object $\Dec_0X$ can be pictured as 
\[ 
\xymatrix{ 
X_0 & \ar[l]_-{d_0} X_1 \ar@<-2.5ex>[r]^-{s_0} & X_2 
\ar@<-2.5ex>[l]_-{d_0} \ar[l]_-{d_1} \ar@<-2.5ex>[r]^-{s_0} \ar@<-5ex>[r]^-{s_1} 
& X_3 \ar@<-5ex>[l]_-{d_0} \ar@<-2.5ex>[l]_-{d_1} \ar[l]_-{d_2} & \cdots } 
\]   
Note that the simplicial identity $d_0d_1 = d_0d_0$ shows 
that $d_0\colon X_1\to X_0$ is an augmentation.  
 
There is an analogous functor $\Dec^0\colon s\cC\to as\cC$ 
given by restriction along the functor $\sigma([0], -)\colon 
\Delta_a\to\Delta$ --- thus $\Dec^0$ is the functor which forgets 
the bottom face and degeneracy map at each level.  The functors 
$\Dec_0$ and $\Dec^0$ are usually called the {\em d\'{e}calage} or {\em shifting} functors.  
More generally we can define functors $\Dec_n\colon s\cC\to as\cC$ 
and $\Dec^n\colon s\cC\to as\cC$ induced by restriction 
along $\sigma(-, [n])\colon \Delta_a\to \Delta$ and 
$\sigma([n], -)\colon \Delta_a\to \Delta$ respectively.       
\end{definition}

The relation between $\Dec_nX$ and $\Dec^nX$ can be 
easily understood through the notion of the opposite simplicial object.  
Let $\tau\colon \Delta\to \Delta$ denote the automorphism 
of $\Delta$ which reverses the order of each ordinal $[n]$, 
or equivalently sends the category $[n]$ to its 
opposite category.  Note that $\tau(\sigma([n], [m])) = 
\sigma(\tau([m]),\tau([n]))$ for any $[n],[m]\in \Delta$.  
If $X$ is a simplicial object then we write $X^o$ for the 
simplicial object obtained by precomposing $X$ with the functor $\tau^\op$.  
The simplicial object $X^o$ is called the 
{\em opposite} simplicial object of $X$ in \cite{Joyal}.  Note that 
$(\Dec_0X)^o = \Dec^0(X^o)$ by the following calculation:
\[
(\Dec_0X)^o([n]) = \Dec_0X(\tau([n]) = X(\sigma([0], 
\tau([n]))) = X(\sigma(\tau([n]), [0])),
\]
since $\tau[0] = [0]$.  It follows that $(\Dec_nX)^o = \Dec^n(X^o)$ for any $n\geq 0$.  

There are canonical comonads underlying the functors 
$\Dec_0$ and $\Dec^0$, when these functors are thought of 
as endofunctors on $s\cC$ by forgetting augmentations.  
As is well known, $[0]$ determines a monoid in $\Delta$ whose 
multiplication is given by the canonical map 
$[1]\to [0]$.  This monoid is universal in a certain precise sense 
(see Proposition 5.1 in Chapter VII of \cite{CWM}).   

The monoid $[0]$ determines a corresponding comonoid in 
$\Delta^\op$ which in turn induces by composition the two comonads $\Dec_0$ and 
$\Dec^0$ in $s\cC$.  The counit of the comonad $\Dec_0$ is induced by the 
natural transformation $[n]\to \sigma([0], [n])$ and hence is 
given on a simplicial object $X$ by the simplicial map 
$\Dec_0 X\to X$ which in degree $n$ is the last face map 
$d_{n+1}\colon X_{n+1}\to X_n$.  We will write 
$d_\last\colon \Dec_0X\to X$ for this map.  

Likewise, the counit of the comonad $\Dec^0$ is induced by the 
natural transformation $[n]\to \sigma([n], [0])$ and hence is given 
on a simplicial object $X$ by the simplicial map $\Dec^0 X\to X$ 
which in degree $n$ is the first face map $d_{0}\colon X_{n+1}\to X_n$.  We will write 
$d_\first\colon \Dec^0X\to X$ for this map.    

When $\Dec_0$ and $\Dec^0$ are regarded as endofunctors on $s\cC$, we see that the functors 
$\Dec_n$ and $\Dec^n$ (also thought of as endofunctors on $s\cC$) are given by $\Dec_n = (\Dec_0)^n$ and 
$\Dec^n = (\Dec_0)^n$ respectively.  

\subsection{Contractibility of the d\'{e}calage functor} 
\label{sec:contractible}
It is an important fact that $\Dec_0X$ and $\Dec^0X$ are not 
just augmented simplicial objects, they are actually contractible augmented 
simplicial objects in the following sense. 

Recall that the augmentation map of an augmented simplicial object $\epsilon \colon X\to X_{-1}$ is a  
a deformation retraction if there exists a simplicial map $s\colon X_{-1}\to X$ (with $X_{-1}$ 
is regarded as a constant simplicial object) which is a section 
of the projection $\epsilon$ and is such that $s\epsilon$ is simplicially homotopic 
to the identity map on $X$.     

A sufficient condition for $s\epsilon$ to be simplicially homotopic to the identity map on 
$X$ is that there exist 
for each $n\geq -1$, maps $s_{n+1}\colon X_n\to X_{n+1}$ with $s_0 = s$, which 
act as `extra degeneracies on the right' in the sense that the following identities hold: 

\begin{align*} 
& d_is_n = s_{n-1}d_i\  \text{for}\  0\leq i<n, \\ 
& d_ns_n = \id, \\ 
& s_is_n = s_{n+1}s_i\ \text{for}\ 0\leq i \leq n, 
\end{align*}  

The following definition is standard.  

\begin{definition}
Let $\epsilon\colon X\to X_{-1}$ be an augmented simplicial object in $\cC$.  
By a {\em contraction} of $X$ we will mean  
the data of the section $s\colon X_{-1}\to X$ together with the 
extra degeneracies $s_{n+1}$ as described above.  We will say that 
$X$ is {\em contractible} if it has such a contraction.  
\end{definition} 

A map of contractible augmented simplicial 
objects is a map of the underlying 
augmented simplicial objects which preserves 
the corresponding sections $s$ and the extra degeneracies (as in \cite{Duskin} we will 
sometimes say that such a map is {\em coherent}).  We will write 
$a_cs\cC$ for the category of contractible augmented simplicial objects and coherent maps. 

Given the data of such a collection of maps $s_{n+1}$ as above, we define maps 
$h_i\colon X_n\to X_{n+1}$ by the formula 
\[
h_i  = s_0^{n-i}s_{n+1}d_0^{n-i}.  
\]
It is easy to check that the maps $h_i$ satisfy the conditions 
(i)--(iii) in Definition 5.1 of \cite{May-SOAT}.  The $h_i$ then piece together 
to define a simplicial homotopy $h\colon X\otimes \Delta[1]\to X$ from $s\epsilon$ to 
the identity on $X$, analogous to Proposition 6.2 in \cite{May-SOAT}.  
Here, if $K$ is a simplicial set, $X\otimes K$ denotes the tensor for the usual 
structure of $s\cC$ as a simplicially enriched category, so that $X\otimes K$ has $n$-simplices given by 
\begin{equation}
\label{simplicial structure}
(X\otimes K)_n = \coprod_{k\in K_n} X_n. 
\end{equation}
In degree 
$n$, the map $h\colon X\otimes \Delta[1]\to X$ is given by 
$d_{i+1}h_i\colon (X_n)_\a \to X_n$ on the summand 
$(X_n)_\a$ of $(X\otimes \Delta[1])_n$ corresponding to the map $\a\colon [n]\to [1]$  
determined by $\a^{-1}(0) = [i]$.
We summarize this discussion in the following lemma.  

\begin{lemma} 
Let $\epsilon\colon X\to X_{-1}$ be a contractible augmented simplicial object in $\cC$.  
Then there is a simplicial homotopy $h\colon X\otimes \Delta[1]\to X$ in $s\cC$ between $s\epsilon$ and $1_X$.  
\end{lemma} 

Clearly, the degeneracy $s_{n+1}\colon X_{n+1}\to X_{n+2}$ for $n\geq 0$ equips 
$\Dec_0 X$ with an extra degeneracy in the above sense.  Therefore we have the following well known result.  
\begin{lemma} 
\label{lem:def retract}
For any simplicial object $X$ in $\cC$, the augmentation $d_0\colon \Dec_0 X\to X_0$ 
is a deformation retract.  An analogous statement is true for $\Dec^0X$.  
\end{lemma}

A prime example where simplicial objects with extra degeneracies appear 
is in the construction of simplicial comonadic resolutions.  Suppose that $L$ is a comonad 
on a category $\cC$, and $X$ is an object of $\cC$.  Then, as is well known, $L$ determines an 
augmented simplicial object $L_*X$ whose object of $n$-simplices is $L^nX$ and whose 
face and degeneracy maps are defined by 
\[
d_i = L^i\epsilon L^{n-i}, s_i = L^i\delta L^{n-i-1} 
\]
respectively, where $\epsilon \colon L\to 1$ denotes the counit and $\delta\colon L\to L^2$ denotes the 
comultiplication of the comonad.  Suppose that there exists a section $\sigma\colon X\to LX$ 
of the counit $\epsilon_X\colon LX\to X$.  Then $\sigma$ determines extra degeneracies 
$s_{n+1}\colon L^nX\to L^{n+1}X$ given by $s_{n+1} = L^n\sigma$ (see for example \cite{Weibel}).  It follows from the discussion 
above that there is a simplicial homotopy $h\colon L_*X\otimes \Delta[1]\to L_*X$ in $s\cC$ between 
$\sigma\epsilon$ and the identity on $L_*X$.

\section{The total d\'{e}calage and the total simplicial set functors}  

In this section we will recall some of the main properties of Illusie's total d\'{e}calage 
functor $\Dec$ \cite{Illusie} and its right adjoint, the Artin-Mazur total simplicial set functor \cite{AM}.  For more details 
the reader should refer to the excellent discussion of these functors and their 
properties in the papers \cite{CR,CR2}.  In this section we will 
mainly be interested in the case where $\cC = \Set$.  We begin therefore by 
explaining our notations and conventions for bisimplicial sets (which follows 
closely the presentation in \cite{JT1}).      

If $X\in \SS$ is a bisimplicial set then we will say that $X_{m,n} = X([m],[n])$ has horizontal degree $m$ and 
vertical degree $n$.  We write $\SS$ for the category of bisimplicial sets.  We say a {\em simplicial space} 
is a simplicial object in $\S$.  There are two ways in which we can regard 
a bisimplicial set $X$ as a simplicial space.  On the one hand, we can 
define $X_m$ to be the simplicial set with $n$-simplices 
$(X_m)_n = X_{m,n}$.  Thus we regard $X$ as a horizontal 
simplicial object with vertical simplicial sets.  
On the other hand we can define $X_n$ to be the simplicial set with 
$m$-simplices $(X_n)_m = X_{m,n}$.  Thus we regard $X$ as a vertical simplicial object 
with horizontal simplicial sets.  

Each of these two ways of viewing a bisimplicial set as a 
simplicial space leads to a simplicial enrichment of $\SS$, using the 
canonical simplicial enrichment of $s\S$ mentioned earlier.  
If we view bisimplicial sets as a horizontal simplicial objects in $\S$, then $\SS = s\S$ 
is equipped with the structure of a simplicial enriched category for which the tensor 
$X\otimes_1 K$, for $X$ a bisimplicial set and $K\in \S$, has vertical simplicial set 
of $m$-simplices given by (see~\eqref{simplicial structure}) 
\[
(X\otimes_1 K)_m = \coprod_{k\in K_m} X_m = X_m\times K_m, 
\]
so that the set of $(m,n)$-bisimplices of $X\otimes_1 K$ is $(X\otimes_1 K)_{m,n} = X_{m,n}\times K_m$.  
In other words, 
\[
X\otimes_1 K = X\times p_1^*K, 
\]
where $p_1\colon \Delta\times \Delta\to \Delta$ denotes projection onto 
the first factor.  The simplicial enrichment is then defined by the formula 
\[
\Hom_1(X,Y) = i_1^*(Y^X), 
\]
where $i_1\colon \Delta\to \Delta\times \Delta$ denotes the right 
adjoint to $p_1$, so that $i_1([n]) = ([n],[0])$.  

Similarly,   
if we view $X\in \SS$ as a vertical simplicial object, then the tensor 
$X\otimes_2 K$ is given by 
\[
X\otimes_2 K = X\times p_2^*K 
\]
where $p_2\colon \Delta\times \Delta\to \Delta$ denotes projection 
onto the second factor.  The simplicial enrichment is defined by the formula 
\[
\Hom_2(X,Y) = i_2^*(Y^X), 
\]
where $i_2\colon \Delta\to \Delta\times \Delta$ denotes 
right adjoint to $p_2$, defined by $i_2([n]) = ([0],[n])$.  
 
Following Joyal we will say that a bisimplicial set $X$ is 
{\em row augmented} if there is a map $X\to p_1^*K$ 
in $\SS$ for some simplicial set $K$, and we will say that 
$X$ is {\em column augmented} if there is a map $X\to p_2^*K$ in $\SS$ for 
some simplicial set $K$.

With these conventions understood we can describe Illusie's total d\'{e}calage functor \cite{Illusie}.  
The simplicial comonadic resolution of $\Dec_0$ gives rise to a functor $\Dec\colon \S\to s\S$ 
which sends a simplicial set $X$ to the simplicial space $\Dec X$ which in degree $n$ 
is the simplicial set 
\[
\Dec_n X = (\Dec_0)^n X.  
\]
Here we are thinking of $\Dec X$ as a vertical simplicial object in $\S$ with horizontal 
simplicial sets.  The set of $(m,n)$-bisimplices of the bisimplicial set $\Dec X$ 
is $(\Dec X)_{m,n} = X_{m+n+1}$.  
The horizontal and vertical face operators 
$d_i^h\colon (\Dec\, X)_{m+1,n}\to (\Dec\, X)_{m,n}$ and 
$d_i^v\colon (\Dec\, X)_{m,n+1}\to (\Dec\, X)_{m,n}$ are given by 
$d_i^h = d_i\colon X_{m+n+2}\to X_{m+n+1}$ and 
$d_i^v = d_{m+i+1}\colon X_{m+n+2}\to X_{m+n+1}$ respectively.  
There are similar formulas for the horizontal and vertical degeneracy operators.  
Note that if we regard the bisimplicial set 
$\Dec X$ as a horizontal simplicial object with vertical simplicial sets then 
$\Dec X$ is the simplicial comonadic resolution of $X$ by $\Dec^0$.  The following Lemma is 
straightforward.    

\begin{lemma} 
\label{lem:dec}
The functor $\Dec\colon \S\to \SS$ is given by restriction along the ordinal sum map 
$\sigma \colon \Delta\times \Delta\to \Delta$ so that      
\[ 
\Dec\, X([m],[n]) = X(\sigma([m], [n])) = X_{n+m+1} 
\] 
for $X\in \S$.  
\end{lemma} 
 
We see from this Lemma that in fact $\Dec X$ is the restriction of a bi-augmented simplicial object 
in the sense that the functor $\Dec\colon \Delta^\op\times \Delta^\op\to \cC$ extends canonically to a functor 
$(\Delta_a^\op\times \Delta^\op)\cup (\Delta^\op\times \Delta_a^\op)\to \cC$.  
Therefore $\Dec X$ is both row and column augmented.  The row augmentation $\epsilon_r\colon \Dec X\to p_1^*X$ 
is given by the map $d_\last\colon \Dec_0 X\to X$, while the column augmentation $\epsilon_c\colon \Dec X\to p_2^*X$ is 
given by the map $d_\first\colon \Dec^0X\to X$.

Suppose that $X$ is a simplicial set, and regard $\Dec X$ as a (vertical) simplicial space 
whose rows are the simplicial sets $\Dec_n X$ for $n\geq 0$.  Then the functor 
$\disc\colon \S\to \SS$ which sends a simplicial set $K$ to the constant simplicial 
space whose rows are $K$, has a left adjoint $\pi_0\colon \SS\to \S$.  Note that the 
functor $\disc$ is nothing other than the functor $p_1^*\colon \S\to \SS$.  It is easy 
to compute the value of $\pi_0\Dec X$ as the next lemma shows.  
\begin{lemma}
\label{pi_0 of Dec} 
For any simplicial set $X$, we have $\pi_0\Dec X = X$.  
\end{lemma} 
\begin{proof} 
The value of the functor $\pi_0\colon \SS\to \S$ on a bisimplicial set $Y$ is given by 
the coequalizer 
\[
\pi_0Y = \coeq(\xymatrix{Y_1 \ar@<0.5ex>[r]^-{d_0} \ar@<-0.5ex>[r]_-{d_1} & Y_0 } ),  
\]
where we have written $Y_n$ for the $n$-th row of $Y$.  
Since colimits in $\S$ are computed pointwise we see that the set of $n$-simplices 
$\pi_0(Y)_n$ is the set of path components of the $n$-th column of $Y$.  Therefore 
the set of $n$-simplices of $\pi_0\Dec X$ is the set of path components of $\Dec^nX$, 
which we have seen is equal to $X_n$.  One can further check that this isomorphism 
is compatible with face and degeneracy maps so that we get the identification $\pi_0 \Dec X = X$.  
\end{proof} 

Thus when $X = \Delta[n]$ we have an isomorphism $\pi_0\Dec \Delta[n] = \Delta[n]$.  
In fact, we will show that more is true: we will see that the row and column augmentation 
maps $\Dec\Delta[n]\to p_1^*\Delta[n]$ and $\Dec \Delta[n]\to p_2^*\Delta[n]$ are simplicial 
homotopy equivalences.        
To see this we will need the following result.  

\begin{lemma} 
For any $n\geq 0$ there are sections 
of the maps $d_\last\colon \Dec_0\Delta[n]\to \Delta[n]$ and $d_\first\colon \Dec^0\Delta[n]\to \Delta[n]$.  
\end{lemma} 

\begin{proof}
Consider the map $\sigma_r\colon \Delta[n]\to \Dec_0\Delta[n]$ given in degree 
$m$ by 
\[
\sigma_r(x) = s^{n+1}\sigma(x,[0]) 
\]
for $x\colon [m]\to [n]$.  Clearly this is natural in $x$.  It is also a section of $d_\last$ by the 
following calculation.  Since  
$d_\last \sigma_r(x) = d_{m+1} s^{n+1}\sigma(x,[0])  =   s^{n+1}\sigma(x,[0])d^{m+1}$, then for any $i\in [m]$ we have 
\begin{align*} 
s^{n+1}\sigma(x,[0])d^{m+1}(i) & = s^{n+1}\sigma(x,[0])(i) \\ 
& = s^{n+1}(x(i)) \\ 
& = x(i),  
\end{align*}
so that $d_\last\sigma_r = \sigma_r$.  In a completely analogous way one can define 
a section of $d_\first$.     
\end{proof} 

The section $\sigma_r\colon \Delta[n]\to \Dec_0\Delta[n]$ is a section of the counit 
$d_\last\colon \Dec_0\Delta[n]\to \Delta[n]$ and hence defines an extra degeneracy 
of the simplicial comonadic resolution $\Dec\Delta[n]$ of $\Delta[n]$.  As discussed 
in Section~\ref{sec:contractible} above, this 
means that the map $\sigma_r\colon \Delta[n]\to \Dec\Delta[n]$ exhibits $\epsilon_r\colon \Dec\Delta[n]
\to p_1^*\Delta[n]$ as a deformation retraction: thus $\epsilon_r\sigma_r$ is the identity on $p_1^*\Delta[n]$ 
and there is a simplicial homotopy between $\sigma_r\epsilon_r$ and the identity on 
$\Dec\Delta[n]$.  Here we are viewing $\Dec\Delta[n]$ as a vertical simplicial object 
with horizontal simplicial sets and thus the simplicial homotopy is a map 
$h\colon \Dec\Delta[n]\otimes_2\Delta[1]\to \Dec\Delta[n]$.  Completely 
analogous statements apply for the section $\sigma_c$.  We summarize this discussion in the 
following lemma.  

\begin{lemma} 
\label{lem:simplicial homotopy}
There are maps $\sigma_r\colon p_1^*\Delta[n]\to \Dec\, \Delta[n]$ 
and $\sigma_c\colon p_2^*\Delta[n]\to \Dec\, \Delta[n]$ in 
$\SS$ such that the following are true: 
\begin{enumerate}
\item $\sigma_r$ is a section of $\epsilon _r\colon \Dec\, \Delta[n]\to p_1^*\Delta[n]$ 
and $\sigma_c$ is a section of $\epsilon_c\colon \Dec\, \Delta[n]\to p_2^*\Delta[n]$, 

\item there is a simplicial homotopy $h\colon \Dec\, \Delta[n]\otimes_2 
\Delta[1]\to \Dec\, \Delta[n]$ from $\sigma_r\epsilon_r$ to the identity, 

\item there is a simplicial homotopy $k\colon \Dec\, \Delta[n]\otimes_1 
\Delta[1]\to \Dec\, \Delta[n]$ from $\sigma_c\epsilon_c$ to the identity.  
\end{enumerate}
\end{lemma}
  
From the description of $\Dec$ in Lemma~\ref{lem:dec} above it is clear that $\Dec$ has both a left and right adjoint.  
The left adjoint of $\Dec$ is related to the notion of the join of simplicial sets.  
The right adjoint to $\Dec$ is denoted 
$T\colon \SS\to \S$, it was introduced in \cite{AM} where it was called the  
{\em total simplicial set functor}.  It is also known as the {\em Artin-Mazur codiagonal}.     
It has the following explicit description: if 
$X$ is a bisimplicial set then the set $(TX)_n$ of 
$n$-simplices of the simplicial set $TX$   
is given by the equalizer of the diagram
\begin{equation}
\label{equalizer2} 
 (T X)_n\to  \prod^n_{i=0} X_{i,n-i}\rightrightarrows \prod^{n-1}_{i=0} X_{i,n-i-1} 
\end{equation} 
where the components of the two maps are defined by the composites 
\[ 
\prod^n_{i=0}X_{i,n-i} \stackrel{p_i}{\to} X_{i,n-i}\stackrel{d_0^v}{\to} X_{i,n-i-1} 
\] 
and 
\[ 
\prod^n_{i=0} X_{i,n-i} \stackrel{p_{i+1}}{\to} X_{i+1,n-i-1} \stackrel{d_{i+1}^h}{\to} X_{i,n-i-1}. 
\] 
The face maps $d_i\colon (T X)_n\to (T X)_{n-1}$ are given by 
\[ 
d_i = (d_i^v p_0,d_{i-1}^v p_1,\ldots, d_1^v p_{i-1},d_i^hp_{i+1},d_i^h p_{i+2},\ldots, d_i^h p_n) 
\] 
while the degeneracy maps $s_i\colon (T X)_n \to (T X)_{n+1}$ are given by 
\[ 
s_i = (s_i^vp_0,s_{i-1}^vp_1,\ldots, s_0^vp_i,s_i^hp_{i+1},\ldots, s_i^h p_n). 
\] 
The unit map $\eta\colon X\to T\Dec\, X$ of the adjunction $\Dec\dashv T$ is given by the map 
\begin{equation} 
x\mapsto (s_0(x),s_1(x),\ldots, s_n(x))  \label{unit map}
\end{equation}
in degree $n$ (see \cite{CR}).  
In general it is rather difficult to give a simple description of the simplicial set $TX$ for an arbitrary 
bisimplicial set $X$.  When $X$ is constant however, we have the following well-known result.  
\begin{lemma}
\label{calculation for constant object} 
Let $X$ be a simplicial set.  Then there are isomorphisms $Tp_1^*X = Tp_2^*X = X$, natural in $X$. 
\end{lemma} 
\begin{proof} 
Observe that the functor $Tp_1^*$ is right adjoint to the functor $\pi_0\Dec$.  Lemma~\ref{pi_0 of Dec} 
implies that the functor $\pi_0\Dec$ is the identity on $\S$, from which it follows that there is an isomorphism $Tp_1^*X = X$, natural 
in $X$.  The other statement is proven in an analogous fashion.  
\end{proof}

\section{The generalized Eilenberg-Zilber theorem for simplicial sets} 
\label{CR sec}

Our goal in this section is to present an elementary proof of Theorem~\ref{CR thm}.  
Recall that this theorem states that there is a weak equivalence 
\begin{equation} 
\label{CRmap}
dX\to T X  
\end{equation}
of simplicial sets, natural in $X$.  
As mentioned earlier, the proof of Theorem~\ref{CR thm} in \cite{CR} is rather lengthy, 
and so it is of interest to have a 
simpler approach.  
We will describe here another proof, which we think is fairly elementary (as mentioned 
in the Introduction, the forthcoming book \cite{JT2} of Joyal and Tierney contains 
another proof, which proceeds along different lines).       
  
We begin by describing the map~\eqref{CRmap}.  This map is obtained from the 
map of cosimplicial bisimplicial sets
\begin{equation} 
\label{dX -> TX} 
\Dec\Delta\to (p_1^*\Delta\times p_2^*\Delta)\delta 
\end{equation}
by applying the functor $\SS(-,X)\colon c\SS\to \S$.  Here $(p_1^*\Delta\times p_2^*\Delta)\delta$ 
is the cosimplicial bisimplicial set which in degree $n$ is $p_1^*\Delta[n]\times p_2^*\Delta[n]$.  
The map~\eqref{dX -> TX} in degree $n$ is the canonical map 
induced by the row and column augmentations $\epsilon_r\colon \Dec\Delta[n]\to p_1^*\Delta[n]$ 
and $\epsilon_c\colon \Dec\Delta[n]\to p_2^*\Delta[n]$ respectively.  Note that it is possible to describe 
the map $dX\to TX$ much more explicitly at the level of simplices (see \cite{CR}) but we will not need this.

The proof of Theorem~\ref{CR thm} that we shall give essentially boils down 
to the well known fact that the diagonal functor 
$d\colon \SS\to \S$ sends level-wise weak equivalences of bisimplicial sets 
to weak equivalences of simplicial sets. 
In other words, if $f\colon X\to Y$ is a map in $\SS$ such 
that the map $f_n\colon X_n\to Y_n$ on $n$-th rows is a weak equivalence 
for all $n\geq 0$, then $df\colon dX\to dY$ is also a weak 
equivalence.  Alternatively, if the map $f_n\colon X_n\to Y_n$ on the $n$-th columns 
is a weak equivalence for all $n\geq 0$, then $df\colon dX\to dY$ is a weak equivalence.  
  
Recall that $d$ has a right adjoint $d_*\colon \S\to \SS$ 
(see for instance \cite{GJ} page 222) defined by the formula 
\begin{equation}
\label{eq:d*}
(d_*X)_{m,n} = \S(\Delta[m]\times \Delta[n],X).  
\end{equation}
Using the fact that the diagonal $d$ sends level-wise weak equivalences 
to weak equivalences one can prove (see for instance 
\cite{Moerdijk}) that the counit $\epsilon\colon dd_*K\to K$ of this adjunction is a 
weak equivalence for any simplicial set $K$, and so in particular $dd_*T X\to T X$ is a 
weak equivalence for any bisimplicial set $X$.  
Therefore, since we can factor~\eqref{CRmap} as 
\[ 
dX \to dd_*T X\to T X,  
\]
we see that to prove Theorem~\ref{CR thm} it suffices to prove the following proposition.  
\begin{proposition}
\label{level wise we}
The map $dX\to dd_*T X$ is a weak equivalence for any bisimplicial set $X$.
\end{proposition}  
\begin{proof}   
First note that the underlying map $X\to d_*T X$ of bisimplicial sets
is induced by the following map in $cc\SS$: 
\begin{equation} 
\label{map inducing X -> d_*TX} 
\epsilon_r\times \epsilon_c\colon \Dec\,  \Delta\times \Dec\,  \Delta \rightarrow 
p_1^*\Delta \times p_2^*\Delta,
\end{equation}
where $\Dec\, \Delta$ denotes the cosimplicial 
bisimplicial set $[m]\mapsto \Dec\, \Delta[m]$.  This relies on~\eqref{eq:d*} and the fact that 
$\Dec$ preserves products, together with the observation that the 
map~\eqref{dX -> TX} factors as 
\[
\Dec\Delta \to 
\Dec(\Delta\times \Delta)\delta \to (p_1^*\Delta\times p_2^*\Delta)\delta  
\]
Here the map $\Delta\to (\Delta\times \Delta)\delta$ is the canonical map 
inducing the counit $dd_*\to 1$ of the adjunction $d\dashv d_*$.   
The map~\eqref{map inducing X -> d_*TX} in turn factors as 
\[ 
\Dec\,  \Delta\times \Dec\,  \Delta\xrightarrow{1\times \epsilon_c} 
\Dec\, \Delta\otimes_2 \Delta \xrightarrow{\epsilon_r\times 1} p_1^*\Delta\times p_2^*\Delta,  
\] 
where $\Dec\Delta\otimes_2 \Delta$ denotes the bicosimplicial bisimplicial 
set which in bidegree $(m,n)$ is $\Dec\Delta[m]\otimes_2 \Delta[n]$.  Applying the 
functor $\SS(-,X)\colon cc\SS\to \SS$ we get a pair of maps of bisimplicial sets   
\begin{equation}
\label{map one}
(1\times \epsilon_c)^*\colon \SS(\Dec\, \Delta\otimes_2 \Delta,X)
\to d_*TX 
\end{equation}
and 
\begin{equation}
\label{map two}
(\epsilon_r\times 1)^*\colon X
\to \SS(\Dec\, \Delta\otimes_2 \Delta,X).   
\end{equation}
We will prove that the first map is a row-wise weak equivalence and that the 
second map is a column-wise weak equivalence.  Since $d$ sends level-wise 
weak equivalences to weak equivalences, this is enough to prove that the map 
$dX\to dd_*TX$ is a weak equivalence.  
  
\begin{lemma} 
\label{we lemma one}
The map~\eqref{map one} is a row-wise weak equivalence.  
\end{lemma} 
\begin{proof}
The induced map on the $n$-th row induced by~\eqref{map one} is obtained by applying the functor 
$\SS(-,X)$ to the map of cosimplicial bisimplicial sets 
\[
\Dec\Delta\times \Dec\Delta[n]\to \Dec\Delta\times p_2^*\Delta[n].  
\]
Since $\SS$ is cartesian closed (it is a presheaf category), we see that the 
induced map on rows is given by the map of simplicial sets  
\begin{equation} 
\label{map on rows}
\epsilon_c^*\colon \SS(p_2^*\Delta[n],X^{\Dec\Delta})\to \SS(\Dec\Delta[n],X^{\Dec\Delta}),  
\end{equation}
where $X^{\Dec\Delta}$ denotes the simplicial object in $\SS$ whose 
bisimplicial set of $p$-simplices is given by the exponential  
\[
X^{\Dec \Delta[p]}
\]
in $\SS$.  Note that for any simplicial set $K$, there is a bijection  
\[
\S(K,\SS(\Dec\Delta[n],X^{\Dec\Delta})) = \SS(\Dec\Delta[n],X^{\Dec K}),  
\] 
which is natural in $K$.  Thus there is an isomorphism 
\[
\SS(\Dec\Delta[n], X^{\Dec\Delta})^{\Delta[1]} = 
\SS(\Dec\Delta[n]\times \Dec\Delta[1],X^{\Dec\Delta}), 
\]
where $\SS(\Dec\Delta[n], X^{\Dec\Delta})^{\Delta[1]}$ denotes the 
simplicial path space of $\SS(\Dec\Delta[n],X^{\Dec\Delta})$.  
We will use these remarks to prove that the map~\eqref{map on rows} is a simplicial 
homotopy equivalence.  

Note that the map $\sigma_c\colon p_2^*\Delta[n]\to \Dec\, \Delta[n]$ 
 induces a left inverse $\sigma_c^*$ of the map $\epsilon_c^*$.
Therefore, we need to show that there is a simplicial homotopy
\[
(\sigma_c\epsilon_c)^*\simeq \id.
\]
By the previous description of the 
simplicial path space of 
$\SS(\Dec\Delta[n],X^{\Dec\Delta})$, to 
find such a simplicial homotopy   
it suffices to find a map 
\begin{equation}
\label{1st simp homotopy} 
 \Dec\, \Delta[n]\times \Dec\, \Delta[1]\to  \Dec\, \Delta[n], 
\end{equation}
in $\SS$ making the obvious diagram commute.  
A map as in~\eqref{1st simp homotopy} is given by the following composite: 
\begin{equation}
\label{cosimplicial homotopy one}
\Dec\, \, \Delta[n]\times\Dec\, \, \Delta[1]\xrightarrow{1\times \epsilon_r} 
\Dec\, \Delta[n]\otimes_1 \Delta[1]\xrightarrow{k} \Dec\, \Delta[n], 
\end{equation}
where $k\colon \Dec\, \Delta[n]\otimes_1 \Delta[1]\to \Dec\, \Delta[n]$ is the 
simplicial homotopy between $\sigma_c\epsilon_c$ and $\id$ of Lemma~\ref{lem:simplicial homotopy}.  
It is easy to check that this map satisfies the required commutativity condition.   
\end{proof} 

\begin{lemma} 
The map~\eqref{map two} is a column-wise weak equivalence.  
\end{lemma} 

\begin{proof}
The map induced on the $m$-th column by~\eqref{map two} is obtained by applying the functor 
$\SS(-,X)$ to the map of cosimplicial bisimplicial sets 
\[
\epsilon_r\times 1\colon \Dec\Delta[m]\times p_2^*\Delta\to p_1^*\Delta[m]\times p_2^*\Delta.  
\]
Again, since $\SS$ is cartesian closed, we see that the induced map on columns 
is given by 
\[
\SS(p_1^*\Delta[m],X^{p_2^*\Delta})\to \SS(\Dec\Delta[m],X^{p_2^*\Delta}).  
\]
Our strategy again is to show that this map is a simplicial homotopy equivalence.  
Since $\sigma_r$ is a right inverse of $\epsilon_r$, the map 
$(\sigma_r\times 1)^*$ is a left inverse of $(\epsilon_r\times 1)^*$.
Again, a little calculation shows that the simplicial path space of $\SS(\Dec\Delta[m],X^{p_2^*\Delta})$ 
is given by 
\[
\SS(\Dec\Delta[m]\times p_2^*\Delta[1],X^{p_2^*\Delta}).  
\]
Therefore, to find a simplicial homotopy $(\sigma_r\epsilon_r\times 1)^*\simeq \id$, 
it suffices to find a bisimplicial map 
\[
\Dec\, \Delta[m]\otimes_2 \Delta[1]\to \Delta[m]. 
\]
making the obvious diagram commute.  
Such a map is given for example by the simplicial homotopy 
\[ 
h\colon \Dec\, \, \Delta[m]\otimes_2 \Delta[1]\to \Dec\, \, \Delta[m] 
\] 
between $\sigma_r\epsilon_r$ and $\id$ of Lemma~\ref{lem:simplicial homotopy}. 
This completes the proof of the lemma.   
\end{proof}  
The proof of Proposition~\ref{level wise we} is now complete, by the remark above.   
\end{proof} 

In analogy with the fact that the map $dC\to \Tot\, C$ of the generalized Eilenberg-Zilber Theorem
for a chain complex $C$ is a chain homotopy equivalence \cite{DP,GJ}, one may wonder whether 
the analogous map $dX\to TX$ of simplicial sets is a simplicial homotopy equivalence.  
In some cases this is known, for instance when $X$ is the degree-wise nerve 
$NG$ of a simplicial group $G$, as we will discuss in the next section.  We suspect that 
the map $dX\to TX$ is not a simplicial homotopy equivalence for arbitrary $X$, however it 
seems a little difficult to construct a counter-example to support this.  In this direction we can say 
the following however: there is no map $TX\to dX$ which is natural in $X$.  For 
such a map would be induced in degree $n$ by a map $\Delta[n,n]\to \Dec\, \Delta[n]$, 
natural in $[n]$, which by adjointness would in turn be induced by a map $\Delta[2n+1]\to \Delta[n]$, 
natural in $[n]$.  However it is not hard to see that no such map can exist.    
  
\section{Kan's simplicial loop group construction revisited}   
\label{simp loop group sec}

The classifying complex $\Wbar G$ of a simplicial group $G$ was 
introduced in \cite{EM} (see Section 17 of that paper).  We recall the definition.  

\begin{definition}[Eilenberg-Mac Lane \cite{EM}] 
\label{def:Wbar G}
Let $G$ be a simplicial group.  Then $\Wbar G$ is the simplicial set with a single vertex, 
and whose set of $n$-simplices, $n\geq 1$, is given by 
\[
(\Wbar G)_n = G_{n-1}\times G_{n-2}\times \cdots \times G_0.  
\]
The face and degeneracy maps of $\Wbar G$ are given by the following formulas:  
\[
d_i(g_{n-1},\ldots, g_0) = \begin{cases}
		(g_{n-2},\ldots, g_0) & \text{if}\ i=0, \\
		(d_i(g_{n-1}),\ldots, d_1(g_{n-i+1}),g_{n-i-1}d_0(g_{n-i}),g_{n-i-2},\ldots, g_0) & \text{if}\ 1\leq i\leq n
		\end{cases}
\]
and 
\[
s_i(g_{n-1},\ldots, g_0) = \begin{cases} 
			(1,g_{n-1},\ldots, g_0) & \text{if}\ i=0, \\ 
			(s_{i-1}(g_{n-1}),\ldots, s_0(g_{n-i}),1,g_{n-i-1},\ldots, g_0) & \text{if}\ 1\leq i\leq n.  
			\end{cases}
\]
\end{definition} 

The motivation for the above formula for $\Wbar G$ is perhaps not so clear.  
We will show that there is a very natural `explanation' for the above formula in terms 
of the d\'{e}calage functors.  For this, we first need some background 
on principal twisted cartesian products.   

Recall that a {\em principal twisted cartesian product} (PTCP) with structure group $G$ consists of a 
simplicial set $P$ (the total space) and a simplicial set $M$ (the base space) together 
with a map $\pi\colon P\to M$ and an action of $G$ on $P$  
which is principal in the sense that the diagram 
\[ 
\xymatrix{ 
P\times G \ar[d]_-{p_1}\ar[r] & P \ar[d]^-{\pi} \\ 
P \ar[r]_-{\pi} & M} 
\] 
is a pullback, where $p_1$ denotes projection onto the first factor, 
the top arrow is the action of $G$ on $P$, and $\pi$ denotes the projection to the base.  
Moreover, $\pi\colon P\to M$ is required to have a {\em pseudo-cross section} (on the left), 
i.e.\  a family of sections $\sigma_n$ of the maps 
$\pi_n\colon P_n\to M_n$ for all $n\geq 0$ such that 
$\sigma_{n+1}s_i = s_i\sigma_n$ for all $0\leq i\leq n$ and 
$d_i \sigma_n = \sigma_{n-1}d_i$ for all $0<i\leq n$.  

The simplicial set $\Wbar G$ is a classifying space for PTCPs with structure 
group $G$ in the sense that there is a universal PTCP $WG$ with base space 
$\Wbar G$ with the property that every PTCP $P$ on $M$ with structure group $G$ is induced 
by pullback from $WG\to \Wbar G$ along a map $M\to \Wbar G$, the classifying map 
of $P$.  

In \cite{Duskin} Duskin explained how this classical 
notion of pseudo-cross section has a convenient reformulation in terms of 
$\Dec^0$.  In this reformulation, $\sigma$ is required to 
be a section of the induced map $\Dec^0\pi\colon \Dec^0P\to \Dec^0 M$ 
in the category $a_c\S$ of contractible augmented 
simplicial sets and coherent maps (see Section~\ref{sec:contractible}).  

Since $G$ acts principally on $P$, there is a canonical map of bisimplicial sets 
\[
\cosk_0 P\to NG, 
\]
where $NG$ denotes the bisimplicial set which, when viewed 
as a (vertical) simplicial object in $\S$, has as its object of $n$-simplices the 
(horizontal) simplicial set $NG_n$, i.e.\ the nerve of the group $G_n$.  
Also here $\cosk_0P$ denotes the 
0-coskeleton (or \v{C}ech nerve) of $P$, viewed as an object in 
$\S/M$.  Therefore, $\cosk_0 P$ has as its object of $n$-simplices 
the (horizontal) simplicial set $\check{C}(P_n)$ which is the \v{C}ech 
nerve of the map $\pi_n\colon P_n\to M_n$.  
In degree $n$ the canonical map $\cosk_0 \to NG$ is just the 
canonical map $\check{C}(P_n)\to NG_n$ arising from the principal action of $G_n$ on $P_n$.  

One of the advantages of this reformulation of the notion of 
PTCP is that it allows for a very simple and conceptual 
description of the classifying map of $P$ (we find it hard to believe that 
this description was not known to Duskin).  
We have a commutative diagram 
\[ 
\xymatrix{ 
\Dec^0 P \ar[d] \ar[r] & P \ar[d] \\ 
\Dec^0 M \ar[r] & M. }
\] 
Composing the pseudo-cross section $s\colon \Dec^0 
M\to \Dec^0 P$ with the map $\Dec^0 P\to P$ gives rise to a 
map $\Dec^0 M\to P$ over $M$ which extends 
canonically to a simplicial map  
\[ 
\Dec\,  M\to \cosk_0 P
\] 
between simplicial objects in $\S/M$.  
Here $\Dec^0 M$ is thought of as the 
vertical simplicial set of 0-simplices of the bisimplicial set 
$\Dec\, M$.  We can compose this with the canonical map $\cosk_0 P\to NG$ to obtain a map 
$\Dec\,  M\to NG$.  The adjoint of the map $\Dec\, M\to NG$ is a map 
\[ 
M\to T NG 
\] 
which serves as a classifying map for $P$.  One can go further and show that 
there is a canonical PTCP with base space $TNG$ from which $P$ arises via 
pullback along the above map.  The next result shows that $TNG$ is {\em precisely} 
the classifying complex $\Wbar G$.        
\begin{lemma}[Duskin] 
\label{lem:Wbar G}
The classifying complex functor $\overline{W}$ factors as 
\[ 
\overline{W} = T N,  
\] 
so that $\overline{W}G = T NG$ for any simplicial group $G$.   
\end{lemma}
This factorization of $\Wbar$ is due as far as we know to Duskin, who observed that this factorization persists 
when simplicial groups are replaced by simplicially enriched 
groupoids, i.e.\ the functor $\Wbar\colon \SGpd\to \S$ 
introduced by Dwyer and Kan in \cite{DK} also factors as $\Wbar = TN$ (this last observation 
also appears in the MSc thesis of Ehlers \cite{Ehlers}).
\begin{proof} 
This is an essentially straightforward computation, so we will just give a sketch of the details.  To an $n$-simplex 
of $\Wbar G$ consisting of a tuple    
\[
(g_{n-1},g_{n-2},\ldots, g_0) 
\]
as above, we associate the element $(x_0,x_1,\ldots, x_n)$ of $TNG$, where $x_0 = 1$ and 
\[
x_i = (d_0^{i-1}(g_{n-1}),d_0^{i-2}(g_{n-2}),\ldots, d_0(g_{n-i+1}),g_{n-i})\in (NG_{n-i})_i
\]
for $i\geq 1$.  This sets up a bijection $(\Wbar G)_n = (TNG)_n$ which respects face and degeneracy maps.  
\end{proof} 

It is well known that $\Wbar G$ is weakly equivalent to the simplicial set $dNG$, obtained by applying 
the diagonal functor to the degree-wise nerve $NG$ of the simplicial group $G$.  Of course 
this can be seen as an instance of Theorem~\ref{CR thm} in light of the identification 
$\Wbar G = TNG$, but there are easier proofs, see for example \cite{JL}.  In fact, $\Wbar G$ is simplicially 
homotopy equivalent to $dNG$, the point being that both $\Wbar G$ and $dNG$ are fibrant 
(a proof of the latter fact can be found in \cite{JT3}).  In \cite{T} it is shown via explicit calculation 
that the map $f\colon dNG\to \Wbar G$ 
defined by 
\[
f(h_1,\ldots,h_n) = (d_0(h_1),\ldots, d_0^n(h_n)) 
\]
for $h_i\in G_n$ exhibits $dNG$ as a deformation retract of $\Wbar G$.  
There is a further relationship between $dNG$ and $\Wbar G$ (see \cite{BH}): after passing to geometric realizations 
there is an {\em isomorphism} of spaces $|\Wbar G| = |dNG|$.  It is not clear that 
this isomorphism is induced by a simplicial map however.  It would be interesting 
to give a more conceptual proof of the isomorphism from \cite{BH}.  

There are several advantages of the description of 
$\overline{W}$ in Lemma~\ref{lem:Wbar G} over the 
traditional description.  One such advantage of the present description 
is that it becomes manifestly clear that $\Wbar$ has a left adjoint since both of the 
functors $N$ and $T$ do.

\begin{proposition} 
\label{left adjoint for Wbar}
A left adjoint for the functor $\overline{W} = T N$ is given by the functor 
\[
G = \pi_1 R\,  \Dec\colon \S\to s\Gp,  
\] 
where $R\colon \SS\to s\S_0$ is the left adjoint of the inclusion 
$s\S_0\subset \SS$.  
If $X$ is a simplicial set, then the value of $G$ on $X$ is the simplicial group $GX$ defined by  
\[ 
[n]\mapsto \pi_1(\Dec_n X/X_{n+1}).  
\] 
\end{proposition}

\begin{proof} 
Observe that the functor $R$ is induced by the left adjoint of the inclusion 
$\S_0\subset \S$, i.e.\ the functor which sends a simplicial set $X$ to the 
reduced simplicial set $X/\sk_0X$.  To describe $RX$ for $X$ a bisimplicial 
set whose $n$-th row is $X_n$, we let $\sk_0X$ denote the bisimplicial set whose 
$n$-th row is $\sk_0X_n$, i.e.\ 
the constant simplicial set $[m]\mapsto X_{0,n}$.  Then $R\, X = X/\sk_0X$ so that 
the $n$-th row of $RX$ is $R\, X_n = X_n/X_{0,n}$.  The proposition then follows 
from the fact that $\sk_0\Dec_n X$ is the constant simplicial set $X_{n+1}$.    
\end{proof} 
 
Recall that a simplicial group $G$ is said to be a {\em loop group} for a 
simplicial set $X$ if there is a PTCP $P$ on $X$ with structure group 
$G$ such that $P$ is weakly contractible.  In \cite{Kan1} Kan showed that 
the left adjoint $G\colon \S\to s\Gp$ of the classifying complex functor 
$\Wbar$ had the property that $G(X)$ was a loop group for any reduced 
simplicial set $X$.  We will shortly give a simplified proof of his theorem by  
exploiting the description of $G$ given in Proposition~\ref{left adjoint for Wbar} above.  
Before we do this however we need the following lemmas. 

\begin{lemma} 
\label{lem:F}
Suppose that $X$ is a bisimplicial set whose first column 
is weakly contractible, i.e.\ the simplicial set $[n]\mapsto X_{0,n}$ is weakly contractible.  
Then $X \to R\, X$ is a column-wise weak equivalence.  
\end{lemma} 
\begin{proof} 
For every $m\geq 0$, the vertical simplicial set $(\sk_0 X)_m$ is weakly contractible and so 
$X_m\to X_m/(\sk_0X)_m$ is a weak equivalence of vertical simplicial sets for every $m\geq 0$.  
\end{proof} 

\begin{lemma}
\label{lem:CW}
Let $X$ be a CW complex whose path components are all contractible.  Then $X/X^0$ is a $K(\pi,1)$, 
where $X^0$ denotes the set of vertices of $X$.  
\end{lemma}
\begin{proof} 
$X/X^0$ can be written as a wedge 
\[ 
\bigvee_{\a\in \pi_0(X)}X_\a/X^0_\a, 
\]
where $X_\a$ denote the path components of $X$.  Therefore without loss of generality we can assume that $X$ is a path connected, pointed CW complex.  We then have to show that 
$X/X^0$ is a $K(\pi,1)$.  Choose a strong deformation retraction of $X$ onto a maximal tree $T$ 
in the 1-skeleton $X^1$ of $X$ (see for example I Theorem 5.9 of \cite{Whitehead}).  Then $T/X^0$ is a deformation retract of $X/X^0$ and so 
$X/X^0$ is a wedge of circles, from which the result follows.
\end{proof}
\begin{corollary}
\label{cor:CW}
For any simplicial set $X$, $\Dec_nX/X_{n+1}$ has the weak homotopy type of a $K(\pi,1)$.
\end{corollary}
\begin{proof}
Since $\Dec_n X = \Dec_0 \Dec_{n-1}X$, it is enough to 
prove this for $\Dec_0X/X_1$.  However this follows immediately 
from the Lemma since $\Dec_0X$ deformation retracts onto $X_0$ (see Lemma~\ref{lem:def retract}).  
\end{proof}
With a little extra effort one can use this corollary to construct an explicit isomorphism between $GX$ 
and the simplicial group described by Kan in \cite{Kan1}, however we will not do this here.  
 
We can now give a simple proof of Kan's result from 
\cite{Kan1} that $X\to \overline{W}GX$ is a weak equivalence when $X$ is reduced.  
We will need the following property of the total simplicial 
set functor $T$: as observed in \cite{CR}, since $d$ sends level-wise weak 
equivalences of bisimplicial sets to weak equivalences of simplicial sets, Theorem~\ref{CR thm} implies that this 
property is inherited by $T$.  
As an immediate consequence of this observation, Cegarra and Remedios prove the following: 

\begin{lemma}[\cite{CR}]
\label{cor:CR}
For any simplicial set $X$, the unit map $X\to T \Dec\, X$ is a weak equivalence. 
\end{lemma}

We briefly review the proof of this result from \cite{CR}.  
\begin{proof} 
Cegarra and Remedios observe that the composite of the unit $X\to T\Dec X$ with the map 
$T\Dec X\to Tp_1^*X$ is the identity on $X$, in light of the identification $Tp_1^*X = X$ of 
Lemma~\ref{calculation for constant object}.  Since $T$ sends level-wise weak equivalences to weak equivalences it 
follows that $T\Dec X\to X$ is a weak equivalence and hence the unit map is a weak equivalence.  
\end{proof} 

We are now ready to prove that $GX$ is a loop group for $X$ whenever $X$ is reduced.   

\begin{theorem}[\cite{Kan1}] 
\label{Kan's thm}
Let $X$ be a reduced simplicial set.  Then the unit map  
\[
\eta\colon X\to \overline{W}GX
\]
is a weak equivalence.  Hence $GX$ is a loop group for $X$.  
\end{theorem}

\begin{proof}
The units of the adjunctions $\Dec\dashv T$, $R\dashv U$, and $N_0\dashv \pi_1$ give a factorization of $\eta$   
\[ 
X\to T\Dec X\to T R\, \Dec\,  X\to T N\pi_1 R\, \Dec\,  X  
\]
in $\S$.  The map $X\to T \Dec\,  X$ is a weak equivalence by Lemma~\ref{cor:CR}.  
The maps $T \Dec\,  X\to T R\, \Dec\,  X$ and 
$T R\, \Dec\,  X\to T \pi_1 R\, \Dec\,  X$ are induced by the maps 
\[ 
\Dec\,  X\to R\, \Dec\,  X\ \text{and}\  R\, \Dec\,  X\to  N\pi_1R\, \Dec\,  X 
\] 
in $\SS$. We will show that both of these maps are level-wise weak equivalences.  
The first map is a level-wise weak equivalence by Lemma~\ref{lem:F}, since $\sk_0\Dec X = \Dec^0X$ and $X$ is reduced.  Corollary~\ref{cor:CW} shows that $R\, \Dec_n X$ has the weak homotopy type of a $K(\pi,1)$ 
and so the second map is also a level-wise weak equivalence.    
\end{proof}

\subsubsection*{Acknowledgements} I thank 
Tom Leinster for some useful conversations.  I am grateful to Pouya Adrom, 
John Baez, Tim Porter, Birgit Richter and David Roberts for useful 
comments.  I thank Jean-Louis Loday for the reference to Quillen's paper.  
I am especially grateful to Myles Tierney 
for informing me about his work with Joyal and for several 
very useful conversations, to Sebastian Thomas for pointing out an error 
in an earlier version of this paper and for the reference to his paper, and to 
Jim Stasheff for many detailed comments and suggestions.

\end{document}